\documentclass[11pt]{article}

\usepackage[T1]{fontenc}
\usepackage[utf8]{inputenc}
\usepackage{amsmath,amssymb,amsfonts}
\usepackage{amsthm}
\usepackage{geometry}
\newcommand{\E}{\mathbb E}
\newcommand{\Bin}{\operatorname{Bin}}
\newcommand{\Z}{\mathbb Z}
\newcommand{\R}{\mathbb R}
\newcommand{\N}{\mathbb N}
\newcommand{\At}{\widetilde A}
\newcommand{\Ah}{\widehat A}
\newcommand{\Bh}{\widehat B}
\providecommand{\backmatter}{}
\providecommand{\texorpdfstring}[2]{#1}

\theoremstyle{plain}
\newtheorem{theorem}{Theorem}[section]
\newtheorem{lemma}[theorem]{Lemma}
\newtheorem{proposition}[theorem]{Proposition}
\newtheorem{corollary}[theorem]{Corollary}
\theoremstyle{remark}
\newtheorem{remark}[theorem]{Remark}
\theoremstyle{definition}
\newtheorem{definition}[theorem]{Definition}
\newtheorem{example}[theorem]{Example}

\title{Binomial probabilities at a fixed distance from the mode:\\
size-biasing and the complete asymptotic expansion}

\author{
N.~Elezovi\'c\\
Department of Applied Mathematics,\\
Faculty of Electrical Engineering and Computing,\\
University of Zagreb, 10000 Zagreb, Croatia\\
\texttt{neven.elezovic@fer.hr}
}

\begin{document}
\maketitle

\begin{abstract}
	We study binomial probabilities at a fixed integer distance from the upper mode, in the
	regime where the fractional part of the mean remains visible.  The complete asymptotic
	expansion is obtained to all orders, uniformly for the success probability in compact
	subintervals of $(0,1)$ and for bounded shifts from the mode.

	The main point is structural.  A logarithmic tail in the standard local expansion is exactly
	the binomial size-bias factor, and removing it leaves a pure Appell expansion.  A second,
	symmetric normalisation is governed by an even Appell sequence and recovers the classical
	Stirling prefactor.  The resulting coefficients explain the exact binomial mode rule, recover
	the central-binomial expansions as a special case, and give closed forms for averaged
	oscillating coefficients.
\end{abstract}

\medskip
\noindent
\textbf{Keywords.} Binomial distribution, local limit theorem, mode, asymptotic expansion,
Bernoulli polynomials, Appell sequence, size-biasing, Euler--Maclaurin formula.

\medskip
\noindent
\textbf{2020 Mathematics Subject Classification.} 60C05, 41A60, 11B68, 33B15.

\section{Introduction}

\subsection{The problem}

Let $X\sim\Bin(N,p)$, $0<p<1$, $q=1-p$, and write
\[
	b(m):=b(m;N,p)=\Pr\{X=m\}=\binom{N}{m}p^{m}q^{N-m}.
\]
The De~Moivre--Laplace theorem says that $b(m)\approx(2\pi Npq)^{-1/2}$ when $m$ is at the
mode.  We ask for the complete expansion in this regime, and for what it fixes about the two most
	basic descriptors of the law at its peak: \emph{where} the peak sits (the mode) and
	\emph{how high} it is (the mass there).  Both are elementary quantities, yet the answer is
	surprisingly rigid; and both feed the companion study of the mean absolute deviation, where
	De~Moivre's identity expresses the first absolute moment through precisely this peak mass.

The results are meant for problems in which the binomial law is sampled or compared at its
largest lattice masses, rather than in a central-limit window scaled by $\sqrt N$.  Such
probabilities enter sharp estimates of the binomial mode and modal mass, local approximations to
finite binomial likelihoods, continuity corrections in lattice asymptotics, and exact formulae
for absolute deviations.  They also occur as building blocks in the surrounding series of papers
on total variation statistics: the local mass at the mode is the normalising object behind
De~Moivre's formula, and fixed displacements from the mode appear when one compares neighbouring
lattice cells or sums short windows around the maximum.  In these uses the oscillation of the
fractional part of $Np$ is not an error term; it is part of the coefficient itself.

The question has an arithmetic component which must be kept explicit.
The upper mode of $\Bin(N,p)$ is $\lfloor(N+1)p\rfloor$, an integer, whereas the mean $Np$ is
not.
Set
\begin{equation}\label{eq:setup}
	\nu=\lceil Np\rceil,\quad
	h=h_N=\nu-Np\in[0,1),\quad
	m=\nu+r\quad(r\in\Z\ \text{fixed}),\quad
	t=m-Np=h+r .
\end{equation}
The quantity $h_N$ is the fractional distance from the mean up to the next lattice point.  For
irrational $p$ it is equidistributed in $[0,1)$ and has no limit; for rational
$p=a/b$ it is periodic with period $b$.  Any complete expansion of $b(\nu+r)$ must therefore
carry $h_N$ as a parameter, and the coefficients cannot converge.

This is precisely the situation studied in \cite{paper22} for the displacement $h_N$ alone
(that is, for $r=0$), where it arises from De~Moivre's classical closed form for the mean
absolute deviation.  The present paper answers the natural next question: what happens at
$\nu+r$, and what, exactly, does the integer $r$ contribute?

\subsection{What is known}

Three strands meet here.

\emph{(i) Binomial coefficients near the centre.}  For $p=\tfrac12$ the mass is a binomial
coefficient, and the expansion of $\binom{2x}{x-r}$ in powers of $x^{-1}$, with $r$ a fixed
integer, is classical territory; a complete treatment with explicit Bernoulli coefficients is
given in \cite[\S3]{be-binom}, and the central case $r=0$, together with the Catalan numbers, in
\cite{elezovic_cbc}.  There the shift $r$ enters through power sums $1^{k}+\dots+(r-1)^{k}$,
produced by iterating the Appell difference identity.  But that setting is very special: it takes
$p=\tfrac12$ \emph{and} $N=2x$ even, so that the mean $Np=x$ is an \emph{integer}, whence
$h_N\equiv0$ and there is no oscillation at all.  (For $p=\tfrac12$ and odd $N$ one has instead
$h_N=\tfrac12$, which is precisely the two-mode case; the theorem below covers it, and
Proposition~\ref{prop:mode} locates the two modes at $r=0$ and $r=-1$.)

\emph{(ii) The local mass at an arbitrary real displacement.}  In \cite[Theorem 3.1]{paper22} the
expansion of $\pi_N(t;p)=\Pr\{X=Np+t\}$ is obtained for arbitrary real $t$, uniformly on
compacta, from a lemma on quotients of gamma functions with unequal scalings.  Its coefficients
are
\begin{equation}\label{eq:Ak}
	A_k(t;p)=\frac{(-1)^{k+1}}{k(k+1)}
	\Bigl[B_{k+1}-\bigl(p^{-k}+(-1)^{k+1}q^{-k}\bigr)B_{k+1}(t)\Bigr]
	+\frac{(-1)^{k}t^{k}}{k\,p^{k}} ,
\end{equation}
a Bernoulli part plus what looked, at the time, like an unavoidable elementary tail.

\emph{(iii) The mean absolute deviation.}  De~Moivre's identity
$\E|X-Np|=2\nu q\,b(\nu;N,p)$ with $\nu=\lceil Np\rceil$ turns the mean absolute deviation into
a single local mass, and \cite[Theorem 4.2]{paper22} expands it.  In that passage a cancellation occurs:
the prefactor $\nu$ produced a logarithm whose expansion cancelled the elementary tail of
\eqref{eq:Ak} \emph{term by term}, leaving a pure Appell series.  This term-by-term cancellation
was recorded in \cite[\S8]{paper22}, with the question of whether it reflects a general principle
for Appell sequences; the present paper answers that question.

\medskip\noindent
It is worth saying explicitly why the peak region is \emph{not} covered by the standard local
asymptotics of the binomial: each of them is built for a regime in which the arithmetic of the
mode --- the very thing at issue here --- plays no role.

\begin{itemize}
\item It is not a Stirling or entropy expansion of $\binom Nm$.  Those --- the classical
	$b(N\rho;N,p)\approx e^{-ND(\rho\|p)}\bigl(2\pi N\rho(1-\rho)\bigr)^{-1/2}$ and its
	refinements --- take the \emph{rate} $\rho=m/N$ as the variable and are uniform in $\rho$ on
	compacta of $(0,1)$; see e.g.\ \cite{feller}.  Our $m$ sits at a bounded distance from the
	mode, so $\rho\to p$, and the entropy factor degenerates: what remains is exactly the local
	expansion, and the interesting variable is the \emph{lattice} displacement $t$, not $\rho$.
	(The symmetric normalisation of Theorem~\ref{thm:sym} does recover the entropy prefactor; see
	Remark~\ref{rem:stirling}.)
\item It is not a central-binomial expansion with explicit error bounds.  Those
	(\cite{be-binom,elezovic_cbc} and their ancestors) live at $p=\tfrac12$ with $Np$ an integer,
	so the fractional part is absent by construction; this is precisely the collapse whose loss is
	the subject of \S\ref{sec:checks}.
\item It is not a local Edgeworth or lattice limit theorem.  Those organise the mass by
	$x=(m-Np)/\sqrt{Npq}$ and are designed for $x$ growing with $N$; see \cite{petrov}.  For
	bounded $t$ --- our case --- they are coarser, because they mix orders as soon as $h_N$
	oscillates.  Remark~\ref{rem:edgeworth} shows the two agree where they overlap.
\end{itemize}

\noindent
What is specific here is the organisation \emph{at a fixed integer distance from the upper mode},
with the oscillating $h_N$ retained exactly as a parameter of the coefficients.

\subsection{What is new}

The statements below are short, and so are their proofs, because the right structural device
--- size-biasing, and the two Appell sequences it produces --- makes them so; the content lies
in that device and in two consequences it delivers at once.  The coefficient of $N^{-1}$
reproduces the \emph{exact} mode rule (\S\ref{sec:mode}), and the size-bias identity settles
an open question of the companion paper on the mean absolute deviation (item~1).  We take the
three structural findings in turn.

\medskip\noindent
\textbf{1.  The elementary tail is the size-bias factor.}  With $m=Np+t$,
\[
	\sum_{k\ge1}\frac{1}{N^{k}}\cdot\frac{(-1)^{k}t^{k}}{k\,p^{k}}
	=-\log\Bigl(1+\frac{t}{Np}\Bigr)=-\log\frac{m}{Np},
\]
exactly.  Thus the tail is a single closed factor, and $\bigl(m/(Np)\bigr)b(m;N,p)$ has a purely
Appell expansion.  The mechanism is the size-bias identity
\begin{equation}\label{eq:sizebias-intro}
	m\,b(m;N,p)=Np\;b(m-1;N-1,p),
\end{equation}
which is not new: it is elementary --- both sides equal $Np\binom{N-1}{m-1}p^{m-1}q^{N-m}$ ---
and it is classical.  In the form relevant to the mean absolute deviation it goes back to
Uspensky~\cite[Ch.~IX]{uspensky}, who records
$\E|X-Np|=2Npq\,b(\nu-1;N-1,p)$ directly; the equivalent forms, and the reading of De~Moivre's
formula as ``twice the variance times the density at the mode'', are collected with their history
by Diaconis and Zabell~\cite[pp.~289--290]{diaconis_zabell}, and the underlying coupling
$X^{\mathrm s}\overset{d}{=}1+\Bin(N-1,p)$ is the textbook size bias of the binomial
\cite{agk}.  What is new here is not the identity but its \emph{asymptotic} role.  The elementary
tail of the naive expansion \eqref{eq:Ak} is \emph{exactly} $-\log(m/(Np))$; the size-bias factor
$m/(Np)$ therefore removes it term by term, and the coefficients become pure Appell.  This also
explains the cancellation observed in \cite[\S8]{paper22}:
De~Moivre's prefactor \emph{is} the size-bias factor, so that
$\E|X-Np|=2Npq\,\Pr\{\Bin(N-1,p)=\nu-1\}$ has a pure Appell expansion.  See Theorem~\ref{thm:pure} and
Remark~\ref{rem:sizebias}.

\medskip\noindent
\textbf{2.  Three normalisations, and two Appell sequences.}  Appell purity and the exact symmetry
$b(m;N,p)=b(N-m;N,q)$ are incompatible, and the obstruction is explicit
(Proposition~\ref{prop:asym}).  Restoring the symmetry replaces $B_n(t)$ by the even Appell
sequence
\[
	\Bh_n(t)=\tfrac12\bigl[B_n(t)+B_n(t+1)\bigr],
	\qquad
	\sum_{n\ge0}\Bh_n(t)\frac{z^{n}}{n!}=\frac z2\coth\frac z2\;e^{tz},
\]
and returns exactly the classical Stirling (entropy) prefactor
$\bigl(2\pi m(N-m)/N\bigr)^{-1/2}$; see Theorem~\ref{thm:sym} and Remark~\ref{rem:stirling}.
The naive normalisation is \emph{not} governed by an Appell sequence at all: its coefficients
\eqref{eq:Ak} are a Bernoulli part plus an elementary tail, and no single Appell sequence absorbs
both.  The three are summarised in the following table, in which
$\Xi_k(p)=(-1)^{k+1}p^{-k}+q^{-k}$, $m=Np+t$, and the last column names the summation rule by
which the integer shift $r$ enters (Theorem~\ref{thm:split}).

\begin{center}\small
\begin{tabular}{llll}
\hline
normalisation & prefactor & logarithmic coefficient & rule for $r$\\
\hline
naive \eqref{eq:old}
	& $(2\pi Npq)^{-1/2}$
	& $A_k=\At_k+\dfrac{(-1)^{k}t^{k}}{k\,p^{k}}$
	& ---\\[6pt]
size-biased \eqref{eq:pure}
	& $\dfrac{Np}{m}\,(2\pi Npq)^{-1/2}$
	& $\At_k=\dfrac{(-1)^{k+1}B_{k+1}-\Xi_k(p)B_{k+1}(t)}{k(k+1)}$
	& left endpoint\\[8pt]
entropy \eqref{eq:sym}
	& $\bigl(2\pi\,m(N-m)/N\bigr)^{-1/2}$
	& $\Ah_k=\dfrac{(-1)^{k+1}B_{k+1}-\Xi_k(p)\Bh_{k+1}(t)}{k(k+1)}$
	& trapezoidal\\
\hline
\end{tabular}
\end{center}

\noindent
The size-biased row is pure Appell but not symmetric under $p\leftrightarrow q$; the entropy row
is symmetric but its Appell sequence is $\Bh$, not $B$; and the naive row is neither.

\medskip\noindent
\textbf{3.  The shift $r$ is elementary, and the dichotomy is Euler--Maclaurin.}  All the
Bernoulli content sits in $h$; the integer $r$ enters only through a power sum along
$h,h+1,\dots$, and \emph{which} power sum is decided by the normalisation:
\[
	B_{n}\ \longleftrightarrow\ \text{left endpoint rule},
	\qquad
	\Bh_{n}\ \longleftrightarrow\ \text{trapezoidal rule};
\]
see Theorem~\ref{thm:split}.  This is the Euler--Maclaurin correspondence in elementary form:
$r$ contributes finite power sums, but no new special-function content.

\medskip\noindent
Theorem~\ref{thm:main} assembles the complete expansion.  Section~\ref{sec:mode} shows that the
coefficient of $N^{-1}$ already reproduces the \emph{exact} rule for the mode --- which is not
"the lattice point nearest the mean", the offset being the skewness correction $(p-q)/2$.
Section~\ref{sec:checks} recovers the central-case coefficients of \cite{be-binom} and shows why
the two collapses available there ($h=0$ and $p=q$) are exactly what is lost in general.
Section~\ref{sec:cesaro} averages the oscillation.  All computations have been verified
symbolically and numerically to $60$ digits; see Section~\ref{sec:numerics}.

From the point of view of applications, the most usable forms are Theorem~\ref{thm:main} and
Corollary~\ref{cor:window}.  The theorem gives any fixed local mass $\Pr\{X=\nu+r\}$ with the
oscillating lattice correction retained explicitly; the corollary packages the same information
for a fixed window around the mode.  Proposition~\ref{prop:mode} explains why the first correction
already contains the exact modal threshold, and Proposition~\ref{prop:cesaro} gives the averaged
coefficients when one does not want to track the individual oscillation in $N$.

\subsection{Conventions}

$B_n$ and $B_n(t)$ denote the Bernoulli numbers and polynomials, with
\[
	B_n:=B_n(0),\qquad\text{so in particular}\qquad B_1=-\tfrac12 .
\]
(The convention matters: unlike in \cite{paper22}, a $B_1$-term does occur here; see
Remark~\ref{rem:B1}.)  We use throughout the two structural identities of the Bernoulli Appell
sequence,
\begin{equation}\label{eq:appell}
	B_n(t+1)-B_n(t)=n\,t^{n-1},
	\qquad
	B_n(1-t)=(-1)^{n}B_n(t),
\end{equation}
and the binomial (Appell) translation $B_n(t+s)=\sum_{i=0}^{n}\binom ni B_i(t)\,s^{n-i}$.
Finally, the whole dependence on $p$ will be carried by
\begin{equation}\label{eq:Xidef}
	\Xi_k(p):=\frac{(-1)^{k+1}}{p^{k}}+\frac{1}{q^{k}}
	=\frac{p^{k}+(-1)^{k+1}q^{k}}{p^{k}q^{k}},
	\qquad\text{which obeys}\qquad
	\Xi_k(q)=(-1)^{k+1}\Xi_k(p).
\end{equation}
Reduced by $p+q=1$,
\begin{equation}\label{eq:Xilist}
	\Xi_1=\frac{1}{pq},\quad
	\Xi_2=\frac{p-q}{p^{2}q^{2}},\quad
	\Xi_3=\frac{1-3pq}{p^{3}q^{3}},\quad
	\Xi_4=\frac{(p-q)(1-2pq)}{p^{4}q^{4}},\quad
	\Xi_5=\frac{1-5pq+5p^{2}q^{2}}{p^{5}q^{5}} .
\end{equation}
Note that $\Xi_{2j}=0$ precisely when $p=q=\tfrac12$.

\section{The local mass, and the elementary tail}\label{sec:tail}

We take as given the following, which is \cite[Theorem 3.1]{paper22}.  For real $t$ with
$Np+t\in\{0,1,\dots,N\}$ put $\pi_N(t;p)=\Pr\{X=Np+t\}$.

\begin{theorem}[\cite{paper22}]\label{thm:old}
	Let $K\subset(0,1)$ and $H\subset\R$ be compact.  As $N\to\infty$, uniformly for admissible
	$p\in K$ and $t\in H$ with $Np+t\in\{0,1,\dots,N\}$,
	\begin{equation}\label{eq:old}
		\pi_N(t;p)\sim\frac{1}{\sqrt{2\pi Npq}}\,
		\exp\Bigl(\sum_{k\ge1}\frac{A_k(t;p)}{N^{k}}\Bigr),
	\end{equation}
	with $A_k$ as in \eqref{eq:Ak}.  For every $M$ the remainder after $M$ terms is
	$O(N^{-M-1})$ relative to the leading factor, with an implied constant depending only on
	$M$, $K$ and $H$.
\end{theorem}

The uniformity in $t$ is the point: for a fixed pair $(p,h)$ the set of admissible $N$ is
typically empty, and only a statement uniform in $t$ can be applied to an oscillating
displacement such as $h_N$.  Since $t=h+r$ with $h\in[0,1)$ and $r$ a fixed integer, we have
$t\in[r,r+1)$, a compact set, and Theorem~\ref{thm:old} applies with $H=[r,r+1]$.

Throughout, $K\subset(0,1)$ and $H\subset\R$ are compact, and we abbreviate
\begin{equation}\label{eq:m-def}
	m=Np+t,\qquad N-m=Nq-t .
\end{equation}
The following elementary observation is used repeatedly and, although routine, is what preserves
the relative remainder $O(N^{-M-1})$ across the renormalisations of \S\ref{sec:three}; we
therefore record it.

\begin{lemma}[the prefactors are harmless]\label{lem:prefactors}
	Let $\delta=\min\{\min_{p\in K}p,\ \min_{p\in K}q\}>0$ and let $B=\max_{t\in H}|t|$.  Then for
	all $N>2B/\delta$ and all $p\in K$, $t\in H$,
	\[
		\frac{\delta}{2}\,N\;\le\;m\;\le\;2N,
		\qquad
		\frac{\delta}{2}\,N\;\le\;N-m\;\le\;2N,
		\qquad\text{hence}\qquad
		\Bigl|\frac{t}{Np}\Bigr|\le\frac{B}{\delta N}<\tfrac12,
		\quad
		\Bigl|\frac{t}{Nq}\Bigr|<\tfrac12 .
	\]
	Consequently $Np/m$ and $\sqrt{q\,m/(p\,(N-m))}$ are bounded above and below by positive
	constants depending only on $K$ and $H$, and their logarithms have convergent expansions in
	$N^{-1}$ whose coefficients are bounded uniformly on $K\times H$.
\end{lemma}

\begin{proof}
	$m=Np+t\ge\delta N-B\ge\delta N/2$ once $N\ge2B/\delta$, and $m\le N+B\le2N$; the same for
	$N-m=Nq-t$.  The two ratios are then $|t|/(Np)\le B/(\delta N)$ and likewise for $q$, both
	$<\tfrac12$ for $N>2B/\delta$.
\end{proof}

\begin{proposition}\label{prop:tail}
	For $|t|<Np$,
	\begin{equation}\label{eq:tail}
		\sum_{k\ge1}\frac{1}{N^{k}}\cdot\frac{(-1)^{k}t^{k}}{k\,p^{k}}
		=-\log\Bigl(1+\frac{t}{Np}\Bigr)=-\log\frac{m}{Np},
	\end{equation}
	the series being convergent.  In truncated form, uniformly for $p\in K$ and $t\in H$ and for
	every $M\ge1$,
	\begin{equation}\label{eq:tail-trunc}
		\sum_{k=1}^{M}\frac{1}{N^{k}}\cdot\frac{(-1)^{k}t^{k}}{k\,p^{k}}
		=-\log\frac{m}{Np}+O\bigl(N^{-M-1}\bigr),
	\end{equation}
	with an implied constant depending only on $M$, $K$ and $H$.
\end{proposition}

\begin{proof}
	Put $z=t/(Np)$.  By Lemma~\ref{lem:prefactors}, $|z|\le B/(\delta N)<\tfrac12$ for
	$N>2B/\delta$, so $\sum_{k\ge1}(-1)^{k}z^{k}/k=-\log(1+z)$ converges, which is \eqref{eq:tail};
	and the tail of the series is bounded by $2|z|^{M+1}/(M+1)\le CN^{-M-1}$, which is
	\eqref{eq:tail-trunc}.
\end{proof}

The distinction matters and we insist on it: \eqref{eq:tail} is an \emph{exact, convergent}
identity in $N^{-1}$, whereas \eqref{eq:old} is an \emph{asymptotic} expansion.  It is
\eqref{eq:tail-trunc} that lets the two be combined without ambiguity: one truncates
\eqref{eq:old} at order $M$, subtracts the truncated tail, and replaces it by the exact factor
$Np/m$, the two errors being of the same order $O(N^{-M-1})$.  Nothing is exponentiated
formally.

Thus the elementary tail of \eqref{eq:Ak} is not an additional asymptotic residue: it is one
closed factor, and it depends on $t$ only through the lattice point $m$ itself.  Removing it gives
the natural object.

\begin{theorem}[Pure Appell normalisation]\label{thm:pure}
	Under the hypotheses of Theorem~\ref{thm:old},
	\begin{equation}\label{eq:pure}
	\begin{gathered}
		b(m;N,p)\sim\frac{Np}{m}\cdot\frac{1}{\sqrt{2\pi Npq}}\;
		\exp\Bigl(\sum_{k\ge1}\frac{\At_k(t;p)}{N^{k}}\Bigr),
		\\
		\At_k(t;p)=\frac{(-1)^{k+1}B_{k+1}-\Xi_k(p)\,B_{k+1}(t)}{k(k+1)} ,
	\end{gathered}
	\end{equation}
	with the same uniformity and the same remainder estimate: for every $M$, truncating the
	exponent after $M$ terms leaves a relative error $O(N^{-M-1})$, uniformly on $K\times H$.  The
	coefficients $\At_k$ contain nothing but Bernoulli numbers and Bernoulli polynomials: no power
	of $t$ stands outside a $B_{k+1}(t)$.
\end{theorem}

\begin{proof}
	By \eqref{eq:Ak} and \eqref{eq:Xidef} we have the coefficientwise identity
	\[
		A_k(t;p)=
		\frac{(-1)^{k+1}B_{k+1}-\Xi_k(p)B_{k+1}(t)}{k(k+1)}
		+\frac{(-1)^kt^k}{kp^k}
		=\At_k(t;p)+\frac{(-1)^kt^k}{kp^k}.
	\]
	Fix $M$.  Theorem~\ref{thm:old} gives
	\[
		b(m;N,p)=\frac{1}{\sqrt{2\pi Npq}}\,
		\exp\left(\sum_{k=1}^{M}\frac{A_k(t;p)}{N^k}\right)
		\bigl(1+O(N^{-M-1})\bigr),
	\]
	uniformly on $K\times H$.  Substituting the preceding splitting separates the exponent into
	\[
		\sum_{k=1}^{M}\frac{\At_k(t;p)}{N^k}
		+\sum_{k=1}^{M}\frac{(-1)^kt^k}{kN^kp^k}.
	\]
	By Proposition~\ref{prop:tail}, the second sum is
	$-\log(m/(Np))+O(N^{-M-1})$.  Since $e^{O(N^{-M-1})}=1+O(N^{-M-1})$, this replacement changes
	the relative error only by $O(N^{-M-1})$.  Hence
	\[
		b(m;N,p)=\frac{Np}{m}\cdot\frac{1}{\sqrt{2\pi Npq}}\,
		\exp\left(\sum_{k=1}^{M}\frac{\At_k(t;p)}{N^k}\right)
		\bigl(1+O(N^{-M-1})\bigr).
	\]
	Lemma~\ref{lem:prefactors} shows that $Np/m$ is bounded above and below by positive constants,
	uniformly on $K\times H$, so the error is still a relative error for the whole expression.
	This is the asserted asymptotic expansion.  No divergent series is being exponentiated:
	\eqref{eq:tail-trunc} is used only after a finite truncation has been chosen.
\end{proof}

Since $B_{k+1}=0$ for even $k\ge2$, formula \eqref{eq:pure} splits by parity into
\begin{equation}\label{eq:parity}
	\At_{2j-1}=\frac{B_{2j}-\bigl(p^{1-2j}+q^{1-2j}\bigr)B_{2j}(t)}{2j(2j-1)},
	\qquad
	\At_{2j}=\frac{\bigl(p^{-2j}-q^{-2j}\bigr)B_{2j+1}(t)}{2j(2j+1)} ,
\end{equation}
and these are \emph{exactly} the coefficients $a_k$ of \cite[Theorem 4.2]{paper22}.  Explicitly,
\begin{equation}\label{eq:Atlist}
	\At_1=\frac1{12}-\frac{B_2(t)}{2pq},\qquad
	\At_2=\frac{(q-p)B_3(t)}{6p^{2}q^{2}},\qquad
	\At_3=-\frac1{360}-\frac{(1-3pq)B_4(t)}{12p^{3}q^{3}} .
\end{equation}

\begin{remark}[What the cancellation was]\label{rem:sizebias}
	The size-bias identity
	\begin{equation}\label{eq:sizebias}
		m\,b(m;N,p)=Np\;b(m-1;N-1,p)
	\end{equation}
	is exact for all $N,m,p$ (both sides equal $Np\binom{N-1}{m-1}p^{m-1}q^{N-m}$) and classical; as
	recalled in the introduction, the point is not the identity but that it \emph{is} the elementary
	tail.  Applied to De~Moivre's formula \cite[Proposition 4.1]{paper22}, whose prefactor is
	$\nu=m$, it gives
	\begin{equation}\label{eq:mad-clean}
		\E|X-Np|=2\nu q\,b(\nu;N,p)=2Npq\;\Pr\{\Bin(N-1,p)=\nu-1\} :
	\end{equation}
	the mean absolute deviation is $2Npq$ times a local mass of $\Bin(N-1,p)$, whose expansion is
	pure Appell because the De~Moivre prefactor \emph{is} the size-bias factor and the elementary
	tail \emph{is} $-\log(m/(Np))$.  This is the exact analytic reason for the cancellation raised in
	\cite[\S8]{paper22}.
\end{remark}

\section{Three normalisations, and two Appell sequences}\label{sec:three}

\subsection{Appell normalisation and symmetry}

Since $X\sim\Bin(N,p)$ implies $N-X\sim\Bin(N,q)$, the exact symmetry
\begin{equation}\label{eq:sym-exact}
	b(m;N,p)=b(N-m;N,q),\qquad\text{i.e.}\qquad (t,p)\longmapsto(-t,q),
\end{equation}
holds identically.  The coefficients $A_k$ of \eqref{eq:Ak} respect it: $A_k(t;p)=A_k(-t;q)$.
The pure coefficients $\At_k$ do not, and the defect is explicit.

\begin{proposition}\label{prop:asym}
	For every $k\ge1$,
	\begin{equation}\label{eq:asym}
		\At_k(t;p)-\At_k(-t;q)=\frac{\Xi_k(p)\,t^{k}}{k},
	\end{equation}
	and
	\begin{equation}\label{eq:asym-sum}
		\sum_{k\ge1}\frac{\Xi_k(p)\,t^{k}}{k\,N^{k}}=\log\frac{q\,m}{p\,(N-m)} ,
	\end{equation}
	which is exactly the ratio of the two size-bias prefactors $Np/m$ and $Nq/(N-m)$.
\end{proposition}

\begin{proof}
	By \eqref{eq:appell},
	\[
		B_{k+1}(1+t)-B_{k+1}(t)=(k+1)t^k,
	\]
	and the reflection identity gives
	\[
		B_{k+1}(-t)=B_{k+1}\bigl(1-(1+t)\bigr)
		=(-1)^{k+1}B_{k+1}(1+t)
		=(-1)^{k+1}\bigl[B_{k+1}(t)+(k+1)t^k\bigr].
	\]
	Also $(-1)^{k+1}\Xi_k(q)=\Xi_k(p)$ by \eqref{eq:Xidef}.  Substituting these two identities
	into the definition of $\At_k$,
	\[
		\At_k(-t;q)=\frac{(-1)^{k+1}B_{k+1}-\Xi_k(p)\bigl[B_{k+1}(t)+(k+1)t^{k}\bigr]}{k(k+1)}
		=\At_k(t;p)-\frac{\Xi_k(p)t^{k}}{k},
	\]
	which is \eqref{eq:asym}.  For \eqref{eq:asym-sum}, split
	$\Xi_k(p)/k=\bigl[(-1)^{k+1}p^{-k}+q^{-k}\bigr]/k$ and sum the two logarithmic series:
	\[
		\sum_{k\ge1}\frac{(-1)^{k+1}}{k}\Bigl(\frac{t}{Np}\Bigr)^{k}
		+\sum_{k\ge1}\frac{1}{k}\Bigl(\frac{t}{Nq}\Bigr)^{k}
		=\log\Bigl(1+\frac{t}{Np}\Bigr)-\log\Bigl(1-\frac{t}{Nq}\Bigr)
		=\log\frac{m}{Np}-\log\frac{N-m}{Nq}.
	\]
	Both series converge, and both may be truncated with error $O(N^{-M-1})$ exactly as in
	Proposition~\ref{prop:tail}: by Lemma~\ref{lem:prefactors} we have $|t/(Np)|<\tfrac12$ and
	$|t/(Nq)|<\tfrac12$ for all large $N$, uniformly on $K\times H$, so the second logarithm ---
	which requires $|t|<Nq$ --- is legitimate for the same reason as the first.
\end{proof}

The asymmetry is therefore not an artefact of the calculation: it is the size-bias factor, which by
construction distinguishes $m$ from $N-m$.  If one requires symmetry, the Bernoulli Appell sequence
is replaced by the even sequence introduced below, and the corresponding prefactor is the classical
Stirling normalisation.

\subsection{The symmetric normalisation, and the sequence
\texorpdfstring{$\Bh_n$}{Bhat n}}

\begin{definition}
	Let $\Bh_n$ be the Appell sequence defined, for all $n\ge0$, by
	\begin{equation}\label{eq:Bhat}
		\Bh_n(t):=\tfrac12\bigl[B_n(t)+B_n(t+1)\bigr],
		\qquad\text{equivalently}\qquad
		\sum_{n\ge0}\Bh_n(t)\frac{z^{n}}{n!}=\frac{z}{2}\coth\frac{z}{2}\;e^{tz} .
	\end{equation}
	For $n\ge1$ this may be written as
	\begin{equation}\label{eq:Bhat-alt}
		\Bh_n(t)=B_n(t)+\tfrac n2\,t^{n-1}\qquad(n\ge1),
	\end{equation}
	by the difference identity \eqref{eq:appell}; for $n=0$ the averaged definition gives
	$\Bh_0=B_0=1$, whereas the right-hand side of \eqref{eq:Bhat-alt} would contain the
	meaningless term $0\cdot t^{-1}$.  Only \eqref{eq:Bhat} is used at $n=0$.
\end{definition}

The polynomial \eqref{eq:Bhat-alt} is not new.  It is exactly the reciprocal Bernoulli polynomial
$\mathbf B^{\diamond}_n(t)=\mathbf B_n(t)+\tfrac n2t^{n-1}$ of Kellner~\cite[eq.~(4.1)]{kellner},
introduced there --- with the same reflection relation \eqref{eq:Bhat-parity} --- as the object in
which the $B_1$-term is cancelled; equivalently it is the even part of $B_n$, since
$\Bh_n(t)=\tfrac12\bigl[B_n(t)+(-1)^{n}B_n(-t)\bigr]$ by
$(-1)^{n}B_n(-t)=B_n(t)+nt^{n-1}$ (DLMF~\cite[24.4.5]{dlmf}).  We keep the notation $\Bh_n$ and the
averaged form \eqref{eq:Bhat} for use at $n=0$, and because it is the average of the two
end-of-interval evaluations that the trapezoidal rule of \S\ref{sec:split} will require.

The generating factor is
$\tfrac z2\coth\tfrac z2=\dfrac{z}{e^{z}-1}+\dfrac z2=\sum_{n\ge0}B_{2n}\dfrac{z^{2n}}{(2n)!}$,
which is an \emph{even} function of $z$.  Consequently
\begin{equation}\label{eq:Bhat-parity}
	\Bh_n(-t)=(-1)^{n}\,\Bh_n(t)\qquad(n\ge0),
\end{equation}
a property the Bernoulli polynomials themselves do not have.  The first few are
\[
	\Bh_1=t,\quad \Bh_2=t^{2}+\tfrac16,\quad \Bh_3=t^{3}+\tfrac t2,\quad
	\Bh_4=t^{4}+t^{2}-\tfrac1{30}, 
\]
\[
	\Bh_5=t^{5}+\tfrac53t^{3}-\tfrac t6, \quad
	\Bh_6=t^{6}+\tfrac52t^{4}-\tfrac12t^{2}+\tfrac1{42} .
\]
Two nearby sequences should be kept apart from $\Bh_n$, since both share its parity and both have
an established name.  The \emph{central Bernoulli polynomials}
$B^{c}_n(t)$, generated by $\dfrac{z}{\sinh z}\,e^{tz}$ (Luschny; OEIS A335953), are a
\emph{different} sequence: $B^{c}_2(t)=t^{2}-\tfrac13$, against $\Bh_2(t)=t^{2}+\tfrac16$.  So is
$B_n(t+\tfrac12)$, generated by $\dfrac{z}{2\sinh(z/2)}\,e^{tz}$, which gives
$B_2(t+\tfrac12)=t^{2}-\tfrac1{12}$.  The three agree only for $n\le1$.  We therefore avoid the
name ``central'' and refer to $\Bh_n$ by \eqref{eq:Bhat}.

\begin{theorem}[Symmetric normalisation]\label{thm:sym}
	Under the hypotheses of Theorem~\ref{thm:old},
	\begin{align}
		b(m;N,p)\sim\frac{1}{\sqrt{2\pi\,m(N-m)/N}}\;
		\exp\Bigl(\sum_{k\ge1}\frac{\Ah_k(t;p)}{N^{k}}\Bigr),
		\label{eq:sym}\\
		\Ah_k(t;p)=\frac{(-1)^{k+1}B_{k+1}-\Xi_k(p)\,\Bh_{k+1}(t)}{k(k+1)} ,
		\notag
	\end{align}
	and these coefficients are symmetric:
	\begin{equation}\label{eq:Ahsym}
		\Ah_k(t;p)=\Ah_k(-t;q)\qquad(k\ge1).
	\end{equation}
	Moreover $\Ah_k=\At_k-\Xi_k(p)\,t^{k}/(2k)$.
\end{theorem}

\begin{proof}
	From \eqref{eq:Bhat-alt} (with $n=k+1\ge2$) we have
	$\Bh_{k+1}(t)=B_{k+1}(t)+\tfrac{k+1}2t^{k}$, whence at once
	$\Ah_k=\At_k-\Xi_k(p)\,t^{k}/(2k)$.  Summing and using \eqref{eq:asym-sum},
	\[
		\sum_{k\ge1}\frac{\At_k(t;p)}{N^{k}}
		=\sum_{k\ge1}\frac{\Ah_k(t;p)}{N^{k}}
		 +\frac12\log\frac{q\,m}{p\,(N-m)} ,
	\]
	in the same truncated sense as in Proposition~\ref{prop:asym}.  More explicitly, after
	truncation at order $M$ the discarded part of the logarithmic series is
	$O(N^{-M-1})$, uniformly on $K\times H$, by Lemma~\ref{lem:prefactors}.  Hence multiplication
	by the exponential of one half of the logarithm changes only the displayed prefactor and
	preserves the relative remainder.  Thus Theorem~\ref{thm:pure} gives
	\[
	b(m;N,p)\sim\frac{Np}{m}\sqrt{\frac{q\,m}{p\,(N-m)}}\cdot\frac{1}{\sqrt{2\pi Npq}}\,
		\exp\Bigl(\sum_{k\ge1}\frac{\Ah_k(t;p)}{N^{k}}\Bigr).
	\]
	The prefactor simplifies:
	\begin{align*}
		\frac{Np}{m}\sqrt{\frac{q\,m}{p\,(N-m)}}\cdot\frac{1}{\sqrt{2\pi Npq}}
		&=\frac{N\sqrt{pq}}{\sqrt{m\,(N-m)}}\cdot\frac{1}{\sqrt{2\pi Npq}}\\
		&=\frac{\sqrt N}{\sqrt{2\pi\,m(N-m)}}
		=\frac{1}{\sqrt{2\pi\,m(N-m)/N}},
	\end{align*}
	which is \eqref{eq:sym}.  For \eqref{eq:Ahsym}, combine \eqref{eq:Bhat-parity} with
	\eqref{eq:Xidef}:
	\[
		\Xi_k(q)\,\Bh_{k+1}(-t)
		=(-1)^{k+1}\Xi_k(p)\cdot(-1)^{k+1}\Bh_{k+1}(t)
		=\Xi_k(p)\,\Bh_{k+1}(t).
	\]
	The Bernoulli-number term $(-1)^{k+1}B_{k+1}$ is unchanged under $(t,p)\mapsto(-t,q)$, and the
	last display shows that the only remaining term is unchanged as well.  Therefore
	$\Ah_k(-t;q)=\Ah_k(t;p)$ for every $k$.
\end{proof}

\begin{remark}[The symmetric normalisation is Stirling's]\label{rem:stirling}
	Put $\rho=m/N$.  The prefactor of \eqref{eq:sym} is
	$\bigl(2\pi N\rho(1-\rho)\bigr)^{-1/2}$ --- exactly the prefactor of the classical entropy
	skeleton $b(N\rho;N,p)\approx e^{-N D(\rho\|p)}\bigl(2\pi N\rho(1-\rho)\bigr)^{-1/2}$, where
	$D$ is the Kullback--Leibler divergence.  Moreover
	\[
		\Ah_1(t;p)=\frac1{12}-\frac{\Bh_2(t)}{2pq}
		 =\frac{1}{12}\Bigl(1-\frac1{pq}\Bigr)-\frac{t^{2}}{2pq},
	\]
	so
	\[
		\Ah_1(0;p)=\frac1{12}\Bigl(1-\frac1{pq}\Bigr)
		 =\frac1{12}\Bigl(1-\frac1p-\frac1q\Bigr),
	\]
	the last equality because $p+q=1$ gives $\tfrac1p+\tfrac1q=\tfrac{p+q}{pq}=\tfrac1{pq}$.  In
	the second form this is precisely the $1/(12N)$ Stirling correction of
	$\log N!-\log(Np)!-\log(Nq)!$, term for term.
\end{remark}

For reference, the symmetric coefficients:
\begin{align}
	\Ah_1(t;p)&=\frac1{12}-\frac{\Bh_2(t)}{2pq},
	&
	\Ah_2(t;p)&=\frac{(q-p)\,\Bh_3(t)}{6\,p^{2}q^{2}},
	\label{eq:Ah12}\\[2pt]
	\Ah_3(t;p)&=-\frac1{360}-\frac{(1-3pq)\,\Bh_4(t)}{12\,p^{3}q^{3}},
	&
	\Ah_4(t;p)&=-\frac{(p-q)(1-2pq)\,\Bh_5(t)}{20\,p^{4}q^{4}},
	\label{eq:Ah34}
\end{align}
the same shape as \eqref{eq:Atlist} with $\Bh_{k+1}$ in place of $B_{k+1}$.

\section{Splitting the displacement: Euler--Maclaurin}\label{sec:split}

We now let $t=h+r$ with $r\in\Z$, and ask how $r$ enters.  Introduce the two power sums along
the arithmetic progression $h,h+1,h+2,\dots$ --- the \emph{left endpoint} sum and the
\emph{trapezoidal} sum:
\begin{equation}\label{eq:ST}
	S_k(h,r):=\sum_{j=0}^{r-1}(h+j)^{k},
	\qquad
	T_k(h,r):=\tfrac12h^{k}+\sum_{j=1}^{r-1}(h+j)^{k}+\tfrac12(h+r)^{k}
	\qquad(r\ge1),
\end{equation}
with $S_k(h,0)=T_k(h,0)=0$, and for $r\le-1$
\[
	S_k(h,r):=-\sum_{j=1}^{-r}(h-j)^{k},
	\qquad
	T_k(h,r):=-\Bigl[\tfrac12(h+r)^{k}+\sum_{j=1}^{-r-1}(h-j)^{k}+\tfrac12h^{k}\Bigr].
\]
In all cases $T_k=S_k+\tfrac12\bigl[(h+r)^{k}-h^{k}\bigr]$.

\begin{lemma}\label{lem:faulhaber}
	For every $k\ge1$, every $h$ and every $r\in\Z$,
	\begin{equation}\label{eq:faulhaber}
		S_k(h,r)=\frac{B_{k+1}(h+r)-B_{k+1}(h)}{k+1},
		\qquad
		T_k(h,r)=\frac{\Bh_{k+1}(h+r)-\Bh_{k+1}(h)}{k+1} .
	\end{equation}
\end{lemma}

\begin{proof}
	For $r\ge1$, summing
	$B_{k+1}(h+j+1)-B_{k+1}(h+j)=(k+1)(h+j)^k$ over $0\le j\le r-1$ gives the first formula.
	For $r=0$ both sides vanish.  For $r\le-1$, the same telescoping sum is taken in the opposite
	direction:
	\[
		B_{k+1}(h+r)-B_{k+1}(h)
		=-\sum_{j=1}^{-r}\bigl[B_{k+1}(h-j+1)-B_{k+1}(h-j)\bigr],
	\]
	which is exactly the signed definition of $S_k(h,r)$.  The second formula follows from the
	first and $\Bh_{k+1}=B_{k+1}+\tfrac{k+1}2t^{k}$:
	\[
		\frac{\Bh_{k+1}(h+r)-\Bh_{k+1}(h)}{k+1}
		=S_k(h,r)+\tfrac12\bigl[(h+r)^{k}-h^{k}\bigr]=T_k(h,r).
	\]
\end{proof}

\begin{theorem}[The contribution of $r$]\label{thm:split}
	For every $k\ge1$, every $h$ and every $r\in\Z$,
	\begin{equation}\label{eq:split}
	\begin{aligned}
		\At_k(h+r;p)&=\At_k(h;p)-\frac{\Xi_k(p)}{k}\;S_k(h,r),
		\\
		\Ah_k(h+r;p)&=\Ah_k(h;p)-\frac{\Xi_k(p)}{k}\;T_k(h,r).
	\end{aligned}
	\end{equation}
	Consequently $\At_k(h+r;p)$ and $\Ah_k(h+r;p)$ are polynomials in $r$ of degree at most
	$k+1$, with leading coefficient $-\Xi_k(p)/\bigl(k(k+1)\bigr)$.  The degree is exactly
	$k+1$ unless $\Xi_k(p)=0$ (which occurs precisely for even $k$ and $p=q=\tfrac12$).
\end{theorem}

\begin{proof}
	We prove the first identity; the second is identical with $B$ replaced by $\Bh$ and $S$ by
	$T$.  From the definition of $\At_k$,
	\[
		\At_k(h+r;p)-\At_k(h;p)
		=-\frac{\Xi_k(p)}{k(k+1)}
		\bigl[B_{k+1}(h+r)-B_{k+1}(h)\bigr].
	\]
	By Lemma~\ref{lem:faulhaber}, the bracket is $(k+1)S_k(h,r)$; this gives the first formula in
	\eqref{eq:split}.  Repeating the same calculation with
	$\Ah_k=-\Xi_k(p)\Bh_{k+1}/(k(k+1))$ plus the unchanged Bernoulli-number term gives the
	trapezoidal formula.

	It remains only to record the degree.  The polynomials $B_{k+1}(x)$ and $\Bh_{k+1}(x)$ are
	monic of degree $k+1$, hence
	$B_{k+1}(h+r)-B_{k+1}(h)$ and $\Bh_{k+1}(h+r)-\Bh_{k+1}(h)$ have leading term $r^{k+1}$ as
	polynomials in $r$.  Multiplication by $-\Xi_k(p)/(k(k+1))$ gives the stated leading
	coefficient.  The degree drops precisely when this coefficient vanishes, i.e. when
	$\Xi_k(p)=0$; from \eqref{eq:Xidef} this is equivalent to $k$ even and $p=q=\tfrac12$.
\end{proof}

Thus the integer shift introduces no additional special-function structure.  The Bernoulli
content of the local expansion is carried by the fractional part $h_N$; the shift $r$ enters
through a power sum along $h,h+1,\dots,h+r$.  The normalisation determines the summation rule:
\[
	\text{Bernoulli }B_{k+1}\ \longleftrightarrow\ \text{left endpoint rule } S_k,
	\qquad
	\Bh_{k+1}\ \longleftrightarrow\ \text{trapezoidal rule } T_k .
\]

The same degree count controls the multiplicative coefficients.  In either normalisation, let
$C_n$ denote the exponential (Bell) polynomial in the first $n$ logarithmic coefficients, for
example
\[
	\exp\Bigl(\sum_{k\ge1}\At_k N^{-k}\Bigr)=\sum_{n\ge0}C_nN^{-n}.
\]
Since $\Xi_1(p)=(pq)^{-1}$, the first logarithmic coefficient is genuinely quadratic in $r$.
Thus the monomial $\At_1^{\,n}/n!$ (respectively $\Ah_1^{\,n}/n!$) contributes a nonzero
$r^{2n}$ term to $C_n$.  On the other hand, any product
$\At_{k_1}\cdots\At_{k_j}$ with $k_1+\cdots+k_j=n$ has degree at most
\[
	(k_1+1)+\cdots+(k_j+1)=n+j\le 2n,
\]
with equality only for $k_1=\cdots=k_j=1$.  Hence $C_n$ has degree exactly $2n$ in $r$.  Finally,
the expansion of the exact prefactor
\[
	\frac{Np}{Np+h+r}=\sum_{i\ge0}(-1)^i\Bigl(\frac{h+r}{Np}\Bigr)^i
\]
contributes, at order $N^{-i}$, only degree $i$ in $r$; at total order $N^{-n}$ this is always
below the top degree $2n$ for $n\ge1$.  The prefactor therefore cannot change the leading
$r$-degree.

\begin{remark}[Why the correspondence is forced]
	An Appell sequence with generating factor $A(z)$ inverts a difference operator, and the
	quadrature rule it produces is read off $A$.  For $A(z)=z/(e^{z}-1)$ the operator is the
	forward difference $E-1$, whose inverse is the left endpoint (rectangle) sum; for
	$A(z)=\tfrac z2\coth\tfrac z2=\tfrac z2\cdot\dfrac{E+1}{E-1}$ the operator is
	$(E-1)\big/\tfrac{E+1}2$, whose inverse is the trapezoidal sum.  Thus the two summation rules
	are the Euler--Maclaurin rules associated with the two Appell sequences.  Furthermore the parity
	of $A$ determines the $p\leftrightarrow q$ symmetry through \eqref{eq:Xidef} and
	\eqref{eq:Bhat-parity}.
\end{remark}

\begin{remark}[The $B_1$-term]\label{rem:B1}
	Expanding $B_{k+1}(h+r)=\sum_{i}\binom{k+1}{i}B_i(h)r^{k+1-i}$, the coefficient of $B_1(h)$ in
	$\At_k(h+r;p)$ is $-\Xi_k(p)\,r^{k}/k$.  For a \emph{fixed} $k$ this vanishes exactly when
	$r=0$ or $\Xi_k(p)=0$, the latter occurring precisely for even $k$ at $p=q=\tfrac12$ (when
	$\At_{2j}$ vanishes identically anyway).  Since $\Xi_1(p)=(pq)^{-1}$ is never zero, the full
	expansion contains a $B_1(h)$-term precisely when $r\ne0$.

	This contrasts with \cite[\S6]{paper22}, where no $B_1$-term appears.  There $h=0$ and
	$h\to1^-$ represent the same lattice situation in De~Moivre's formula, so the endpoint values
	of each coefficient must agree.  Here they represent the two different lattice points
	$Np+r$ and $Np+1+r$, whose probabilities need not agree; the cancellation survives only at
	the distinguished point $r=0$.
\end{remark}

\section{The complete expansion}\label{sec:main}

\begin{theorem}[Main]\label{thm:main}
	Fix $R\in\N$.  As $N\to\infty$, uniformly for $p$ in a compact subset of $(0,1)$ and for
	integers $r$ with $|r|\le R$, with $\nu=\lceil Np\rceil$ and $h=h_N=\nu-Np$,
	\begin{equation}\label{eq:main}
		\Pr\{X=\nu+r\}
		\sim\frac{Np}{Np+h+r}\cdot\frac{1}{\sqrt{2\pi Npq}}\;
		\exp\left(\sum_{k\ge1}\frac{1}{N^{k}}
		\left[\At_k(h;p)-\frac{\Xi_k(p)}{k}\,S_k(h,r)\right]\right),
	\end{equation}
	and equivalently
	\begin{equation}\label{eq:main-sym}
		\Pr\{X=\nu+r\}
		\sim\frac{1}{\sqrt{2\pi\,(\nu+r)(N-\nu-r)/N}}\;
		\exp\left(\sum_{k\ge1}\frac{1}{N^{k}}
		\left[\Ah_k(h;p)-\frac{\Xi_k(p)}{k}\,T_k(h,r)\right]\right).
	\end{equation}
	Equivalently, after either exponential is re-expanded multiplicatively and truncated after
	the term $N^{-M}$, the relative remainder is $O(N^{-M-1})$, uniformly in $|r|\le R$ and in
	$p$ on the chosen compact subset of $(0,1)$.
\end{theorem}

\begin{proof}
	Let $K$ be the chosen compact subset of $(0,1)$.  Since $0\le h<1$ and $|r|\le R$, the
	displacement $t=h+r$ belongs to the fixed compact interval $[-R,R+1]$.  Thus
	Theorem~\ref{thm:old}, and hence Theorems~\ref{thm:pure} and~\ref{thm:sym}, apply uniformly
	for $p\in K$ and all allowed $r$.  Also, if
	$\delta=\min_{p\in K}\min(p,q)$, then for $N>2(R+1)/\delta$ the points
	$\nu+r=Np+h+r$ lie in $\{0,1,\dots,N\}$ whenever the probability is nontrivial, so no boundary
	term is involved.

	Substituting $t=h+r$ into Theorem~\ref{thm:pure} gives the size-biased form with coefficients
	$\At_k(h+r;p)$.  The first identity in Theorem~\ref{thm:split} rewrites those coefficients as
	$\At_k(h;p)-\Xi_k(p)S_k(h,r)/k$, which is \eqref{eq:main}.  The same argument, using
	Theorem~\ref{thm:sym} and the second identity in Theorem~\ref{thm:split}, gives
	\eqref{eq:main-sym}.  Since all earlier estimates were uniform on the compact set
	$K\times[-R,R+1]$, the remainder remains uniform in $p$ and $|r|\le R$.
\end{proof}

\begin{remark}[what is, and what is not, being claimed]\label{rem:whatisnew}
	Theorem~\ref{thm:main} is not presented as a new derivation of a local expansion.  The existence of an
	expansion at $t=h+r$ follows at once from \cite[Theorem 3.1]{paper22} by substituting the
	bounded displacement $t=h+r$ --- this is the role of the uniformity in $t$, as noted in
	\S\ref{sec:tail}.  What Theorem~\ref{thm:main} adds is the \emph{structure}: the exact
	identification of the elementary tail as a size-bias factor (Proposition~\ref{prop:tail},
	Theorem~\ref{thm:pure}), the two Appell normalisations and the obstruction between them
	(Proposition~\ref{prop:asym}, Theorem~\ref{thm:sym}), and the reduction of the entire
	$r$-dependence to a quadrature sum (Theorem~\ref{thm:split}).  The present paper is a
	companion to \cite{paper22} and depends on it; Theorem~\ref{thm:old} is quoted, not reproved.
\end{remark}

In multiplicative form, with $\At_k=\At_k(h+r;p)$,
\begin{equation}\label{eq:mult}
	\Pr\{X=\nu+r\}\sim\frac{Np}{\nu+r}\cdot\frac{1}{\sqrt{2\pi Npq}}
	\left[1+\frac{\At_1}{N}+\frac{\At_2+\frac12\At_1^{2}}{N^{2}}
	+\frac{\At_3+\At_1\At_2+\frac16\At_1^{3}}{N^{3}}+\cdots\right],
\end{equation}
the bracket being the exponential (Bell) series and the prefactor $Np/(\nu+r)$ being kept exact.
For comparison with the classical literature it is more convenient to expand the prefactor as
well and to normalise by $(2\pi Npq)^{-1/2}$ alone; doing so, and recalling
$A_1=\At_1-t/p$ from \eqref{eq:Ak}, one obtains to two orders, with $t=h+r$,
\begin{equation}\label{eq:twoorders}
	\Pr\{X=\nu+r\}=\frac{1}{\sqrt{2\pi Npq}}\left[
	1+\frac{1}{N}\left(\frac1{12}-\frac{B_2(t)}{2pq}-\frac tp\right)+O(N^{-2})\right],
\end{equation}
in which the term $-t/p$ is precisely the first order of the prefactor.  For computation
\eqref{eq:main} is the better form; for comparison, \eqref{eq:twoorders}.

\begin{remark}[Consistency with the local Edgeworth expansion]\label{rem:edgeworth}
	Write $\sigma^{2}=Npq$ and let $\mathit{He}_m$ be the \emph{probabilists'} Hermite
	polynomials, normalised so that $\mathit{He}_4(0)=3$ and $\mathit{He}_6(0)=-15$.  Expanding
	\eqref{eq:twoorders},
	\[
		\sqrt{2\pi Npq}\;\Pr\{X=Np+t\}
		=\exp\Bigl(-\frac{t^{2}}{2\sigma^{2}}\Bigr)
		\left[1-\frac{(q-p)\,t}{2\sigma^{2}}
		+\frac{1}{12N}\Bigl(1-\frac1{pq}\Bigr)+O(N^{-2})\right].
	\]
	This is the local Edgeworth expansion
	$\varphi(x)\bigl[1+\tfrac{\gamma_1}{6}\mathit{He}_3(x)+\tfrac{\gamma_2}{24}\mathit{He}_4(x)
	+\tfrac{\gamma_1^{2}}{72}\mathit{He}_6(x)+\cdots\bigr]$, with $x=t/\sigma$,
	$\gamma_1=(q-p)/\sigma$ and $\gamma_2=(1-6pq)/\sigma^{2}$, evaluated at $x=O(N^{-1/2})$:
	the $\mathit{He}_3$-term supplies $-(q-p)t/(2\sigma^{2})$, and the $\mathit{He}_4$- and
	$\mathit{He}_6$-terms supply, through the values at $0$ recorded above,
	\[
		\frac{3(1-6pq)}{24\,\sigma^{2}}-\frac{15(q-p)^{2}}{72\,\sigma^{2}}
		=\frac{pq-1}{12\,Npq},
	\]
	as required.  The two expansions are the same object organised by different variables:
	Edgeworth by $x=t/\sigma$, which is right when $t$ grows with $N$; ours by $t$ itself, which
	is right when $t$ is bounded.  In the present problem ours is the finer one, being a genuine
	expansion in $N^{-1}$ with the oscillating $h_N$ retained exactly, whereas the Edgeworth
	series mixes orders as soon as $h_N$ oscillates.

	Fractional-part terms in lattice expansions are, of course, classical: Esseen's complete
	lattice Edgeworth expansion already contains Bernoulli polynomials of fractional parts
	\cite{esseen}; see also Petrov~\cite[Ch.~VII]{petrov} and Janson~\cite[Rem.~1.4]{janson}.
	The distinction here is that the surviving arithmetic is the fractional part of the mean,
	$h_N=\lceil Np\rceil-Np$, not that of a standardized evaluation point.  Re-expanding the
	lattice Edgeworth series at bounded $t$ recovers such dependence, but not the closed
	all-order form in which the coefficients are $\Xi_k(p)$ times $B_{k+1}(h+r)$.
\end{remark}

\begin{remark}[Range of validity]
	Theorem~\ref{thm:main} requires $r$ bounded.  Since $\At_k(h+r;p)$ has degree $k+1$ in $r$,
	the term $\At_k/N^{k}$ is of size $r^{k+1}/N^{k}$: successive terms decrease only while
	$r=o(N)$, and the first correction is $o(1)$ only while $r=o(\sqrt N)$.  Beyond $r\asymp\sqrt
	N$ one is no longer at a fixed distance from the mode, the natural variable becomes
	$x=t/\sqrt{Npq}$, and the correct organisation is the Edgeworth series, then the Cram\'er
	series, then the saddle point.
\end{remark}

\begin{corollary}[A window around $\nu$]\label{cor:window}
	Fix $R\in\N$.  With $A_1(t;p)=\tfrac1{12}-\dfrac{B_2(t)}{2pq}-\dfrac tp$,
	\[
		\Pr\{|X-\nu|\le R\}
		=\frac{1}{\sqrt{2\pi Npq}}\left[(2R+1)
		+\frac1N\sum_{r=-R}^{R}A_1(h+r;p)+O(N^{-2})\right],
	\]
	and
	\begin{equation}\label{eq:windowsum}
		\sum_{r=-R}^{R}A_1(h+r;p)
		=(2R+1)\left[\frac1{12}-\frac{B_2(h)}{2pq}-\frac hp\right]
		-\frac{R(R+1)(2R+1)}{6\,pq} .
	\end{equation}
\end{corollary}

\begin{proof}
	Sum \eqref{eq:twoorders} over $-R\le r\le R$.  There are $2R+1$ leading terms, and the
	relative remainder $O(N^{-2})$ remains uniform because the number of summands is fixed.  Since
	$B_2(x)=x^2-x+\tfrac16$, we have
	\[
		B_2(h+r)=B_2(h)+(2h-1)r+r^2.
	\]
	Using
	\[
		\sum_{r=-R}^{R}r=0,\qquad
		\sum_{r=-R}^{R}r^{2}=\frac{R(R+1)(2R+1)}{3},
	\]
	we obtain
	\[
		\sum_{r=-R}^{R}B_2(h+r)
		=(2R+1)B_2(h)+\frac{R(R+1)(2R+1)}{3},
	\]
	and also $\sum_{r=-R}^{R}(h+r)=(2R+1)h$.  Substituting these two sums into
	$\sum_r A_1(h+r;p)$ gives \eqref{eq:windowsum}.
\end{proof}

The correction has a dominant negative part of order $R^{3}$, reflecting the local Gaussian
curvature around the mode.

\section{The mode}\label{sec:mode}

\begin{proposition}\label{prop:mode}
	The upper mode of $\Bin(N,p)$ is
	\[
		\lfloor(N+1)p\rfloor=\nu+\lfloor p-h\rfloor
		=\begin{cases}\nu, & h\le p,\\[2pt] \nu-1, & h>p,\end{cases}
	\]
	that is, $r_{\mathrm{mode}}=0$ if $h\le p$ and $r_{\mathrm{mode}}=-1$ if $h>p$.
	When $h=p$ (equivalently $(N+1)p\in\Z$), the lower adjacent point $\nu-1$ is also a mode.
\end{proposition}

\begin{proof}
	For $0\le k\le N-1$,
	\[
		\frac{b(k+1)}{b(k)}
		=\frac{\binom N{k+1}p^{k+1}q^{N-k-1}}{\binom Nk p^kq^{N-k}}
		=\frac{(N-k)p}{(k+1)q}.
	\]
	This ratio is at least one precisely when
	\[
		(N-k)p\ge(k+1)q
		\quad\Longleftrightarrow\quad
		Np\ge kq+p k+q
		\quad\Longleftrightarrow\quad
		k+1\le (N+1)p,
	\]
	because $p+q=1$.  Hence the masses increase up to the last $k$ for which this inequality
	holds and decrease afterwards.  Therefore the largest maximising index is
	$\lfloor(N+1)p\rfloor$, and if $(N+1)p$ is an integer the two adjacent indices
	$(N+1)p-1$ and $(N+1)p$ have equal mass.  Since $Np=\nu-h$,
	\[
		\lfloor(N+1)p\rfloor=\lfloor Np+p\rfloor=\nu+\lfloor p-h\rfloor,
	\]
	with $p-h\in(-1,1)$.  Thus $\lfloor p-h\rfloor=0$ when $h\le p$ and
	$\lfloor p-h\rfloor=-1$ when $h>p$, giving the stated alternatives.  At equality
	$h=p$, one has $(N+1)p=\nu$, so the preceding ratio test gives equality between the two
	adjacent masses $b(\nu-1)$ and $b(\nu)$.
\end{proof}

\begin{remark}[The coefficient of $N^{-1}$ already knows the mode]\label{rem:mode-coeff}
	The exact rule is proved above, by the elementary ratio argument, and that argument is also
	what certifies that no index other than $\nu$ and $\nu-1$ can be maximal.  The following is
	therefore a \emph{consistency check} on the coefficients, not an independent proof.

	By \eqref{eq:Ah12} the exponent of \eqref{eq:main-sym}, to first order, is
	\[
		\frac{1}{N}\left[\frac1{12}\Bigl(1-\frac1{pq}\Bigr)-\frac{t^{2}}{2pq}\right]
		\quad\text{plus the expansion of the prefactor,}
	\]
	and collecting the linear term of the latter one finds the exponent
	$-\bigl(t^{2}-(p-q)t\bigr)/(2Npq)+\text{const}$: a parabola in $t$ whose vertex is at
	\[
		t_{*}=\frac{p-q}{2},
	\]
	\emph{not} at $t=0$.  Of the two adjacent candidates $\nu-1$ and $\nu$ --- that is,
	$t\in\{h-1,h\}$ --- the upper one is closer to $t_{*}$ precisely when $h-t_{*}\le\tfrac12$,
	that is, when
	\[
		h\le\frac{p-q}{2}+\frac12=p .
	\]
	The threshold delivered by the first asymptotic coefficient is therefore the \emph{exact}
	threshold for the upper mode in Proposition~\ref{prop:mode}; equality is exactly the
	two-mode case.

	Thus the mode is not, in general, the lattice point nearest the mean: the criterion is
	$h\le p$, not $h\le\tfrac12$.  The vertex $t_*=(p-q)/2$ is the usual skewness correction
	(already present in the $\mathit{He}_3$-term of Remark~\ref{rem:edgeworth}); the point here is
	that the first local coefficient gives the same correction together with the exact lattice
	rounding rule.  For related discussions of binomial modes and of the mode displacement in
	nearly normal laws see Kaas and Buhrman~\cite{kaas_buhrman}, Haldane~\cite{haldane}, and
	Hall~\cite{hall}.
\end{remark}

\begin{remark}[the symmetry orbit, and its endpoint]\label{rem:orbit}
	In terms of $(h,r)$ the exact symmetry \eqref{eq:sym-exact} reads
	\[
		(h,r,p)\longmapsto(1-h,\,-1-r,\,q)\qquad(0<h<1),
	\]
	so that the reflection centre in $r$ is $r=-\tfrac12$ and the mirror of $r=0$ is $r=-1$.

	At the endpoint $h=0$ the convention $h_N\in[0,1)$ replaces the reflected representative
	$h'=1$ by $h'=0$.  The symmetry is then
	\[
		(0,r,p)\longmapsto(0,\,-r,\,q),
	\]
	which designates the same reflected lattice point as the non-canonical choice
	$(1,-1-r,q)$.  Thus the pair $\{0,-1\}$ of Proposition~\ref{prop:mode} is one symmetry orbit
	for $0<h<1$; the ratio argument, however, is what proves that no other lattice point can be a
	mode.
\end{remark}

\section{Two consistency checks}\label{sec:checks}

\subsection{The exact ratio, and why the prefactor must be there}

For $r\ge1$, telescoping the exact ratio $b(k+1)/b(k)=\dfrac{(N-k)p}{(k+1)q}$ gives
\begin{equation}\label{eq:exactratio}
	\frac{b(\nu+r)}{b(\nu)}=\prod_{j=1}^{r}\frac{(N-\nu-j+1)\,p}{(\nu+j)\,q} .
\end{equation}
Theorem~\ref{thm:main} must reproduce this identically, and the way it does is instructive.  By
Theorem~\ref{thm:split}, and splitting $\Xi_k(p)/k$ into its two logarithmic series as in the
proof of Proposition~\ref{prop:asym},
\[
	\sum_{k\ge1}\frac{\At_k(h+r;p)-\At_k(h;p)}{N^{k}}
	=-\sum_{j=0}^{r-1}\left[\log\Bigl(1+\frac{h+j}{Np}\Bigr)
	   -\log\Bigl(1-\frac{h+j}{Nq}\Bigr)\right].
\]
Exponentiating, and multiplying by the ratio of the two prefactors,
$\bigl(Np/(\nu+r)\bigr)\big/\bigl(Np/\nu\bigr)=\nu/(\nu+r)$,
\[
	\frac{b(\nu+r)}{b(\nu)}
	=\frac{\nu}{\nu+r}\prod_{j=0}^{r-1}\frac{(Nq-h-j)\,p}{(Np+h+j)\,q}
	=\frac{\nu}{\nu+r}\prod_{j=0}^{r-1}\frac{(N-\nu-j)\,p}{(\nu+j)\,q}
	=\prod_{j=1}^{r}\frac{(N-\nu-j+1)\,p}{(\nu+j)\,q},
\]
the last step being the telescoping of the denominator: the product
$\prod_{j=0}^{r-1}(\nu+j)$ equals $\nu(\nu+1)\cdots(\nu+r-1)$, and multiplying by
$\nu/(\nu+r)$ cancels the leading factor $\nu$ while appending the trailing factor $\nu+r$,
\[
	\frac{\nu}{\nu+r}\cdot\frac{1}{\nu(\nu+1)\cdots(\nu+r-1)}
	=\frac{1}{(\nu+1)(\nu+2)\cdots(\nu+r)}
	=\prod_{j=1}^{r}\frac{1}{\nu+j}\,;
\]
the numerators are unchanged by the reindexing $j\mapsto j+1$.  So the series reproduces
\eqref{eq:exactratio} exactly --- and note that the telescoping \emph{requires} the factor
$Np/m$.  Without the size-bias prefactor the identity would not close.  This is an independent
confirmation that \eqref{eq:pure} is the right normalisation, and not merely a convenient one.

The case $r\le-1$ is not merely analogous; since sign conventions in the oriented sums
\eqref{eq:ST} are the easiest place to go wrong, we display it.  Write $r'=-r\ge1$.  Telescoping
$b(k-1)/b(k)=\dfrac{k\,q}{(N-k+1)\,p}$ downwards from $\nu$ gives
\begin{equation}\label{eq:exactratio-neg}
	\frac{b(\nu-r')}{b(\nu)}=\prod_{j=1}^{r'}\frac{(\nu-j+1)\,q}{(N-\nu+j)\,p} ,
\end{equation}
and the same computation as above, now with the \emph{negative} orientation
$S_k(h,-r')=-\sum_{j=1}^{r'}(h-j)^{k}$, reproduces \eqref{eq:exactratio-neg} identically: the
minus sign in \eqref{eq:ST} is exactly what interchanges the roles of $p$ and $q$ in the two
logarithmic series, and the prefactor ratio is now $\nu/(\nu-r')$.  This is the finite-$N$
meaning of the oriented sums.

\subsection{The central case: recovering the coefficients of
\texorpdfstring{\cite{be-binom}}{Buric-Elezovic}}

The expansion of \cite{be-binom} is written for $\binom{2x}{x-\rho}$, so that its shift $\rho$
counts \emph{downwards} from the centre.  In our notation, therefore, $\rho=-r$; we keep the
symbol $\rho$ in this subsection to avoid a clash, and note that the $P_n$ turn out to be even in
$\rho$, so the sign is in the end immaterial.

Take $N=2x$, $p=q=\tfrac12$, $m=x-\rho$.  Then $Np=x$, $\nu=x$, $h=0$, and the displacement is
$t=m-Np=-\rho$.  Theorem~\ref{thm:pure} gives
\[
	4^{-x}\binom{2x}{x-\rho}\sim\frac{x}{x-\rho}\cdot\frac{1}{\sqrt{\pi x}}\;
	\exp\Bigl(\sum_{k\ge1}\frac{\At_k(-\rho;\tfrac12)}{(2x)^{k}}\Bigr),
\]
so that $\binom{2x}{x-\rho}\sim\dfrac{4^{x}}{\sqrt{\pi x}}\sum_{n\ge0}P_n(\rho)\,x^{-n}$, the
$P_n$ being obtained by multiplying the exponential series by $(1-\rho/x)^{-1}$.  Carrying this
out (and writing $r$ for $\rho$ in the displayed polynomials, as in \cite{be-binom}):
\begin{align*}
	P_0&=1, &
	P_1&=-\frac18-r^{2}, \\[2pt]
	P_2&=\frac1{128}+\frac{5r^{2}}{8}+\frac{r^{4}}{2}, &
	P_3&=\frac5{1024}-\frac{91r^{2}}{384}-\frac{35r^{4}}{48}-\frac{r^{6}}{6},
\end{align*}
\[
	P_4=-\frac{21}{32768}+\frac{61r^{2}}{3072}+\frac{161r^{4}}{256}+\frac{7r^{6}}{16}
	     +\frac{r^{8}}{24},
\]
in exact agreement with \cite[Theorem 3.1]{be-binom}, except that the value of $P_1$ printed
there, $-\tfrac18+r^{2}$, is a misprint for $-\tfrac18-r^{2}$: the recursion stated in that same
theorem, and the accompanying values of $P_2,P_3,P_4$, both give the minus sign, as does the
present derivation.  The degrees are $\deg P_n=2n$, exactly
as Theorem~\ref{thm:split} predicts.

Two features of the central case are worth isolating, because they are precisely what is
\emph{lost} for a general $p$.
\begin{itemize}
	\item $h=0$: the argument of the Bernoulli polynomial is an integer, $B_{k+1}(h)$ collapses to
	a Bernoulli \emph{number}, and there is no oscillation.
	\item $p=q$: by \eqref{eq:Xilist} we have $\Xi_{2j}=0$, so \emph{every even logarithmic
	coefficient vanishes in both normalisations}, $\At_{2j}=\Ah_{2j}=0$, and the expansion runs
	over odd $k$ only.  The two normalisations nevertheless remain distinct: since
	$\Xi_{2j-1}(\tfrac12)=2^{2j}\ne0$, they differ in \emph{every} odd --- that is, in every
	surviving --- order, by
	\[
		\Ah_k(t;\tfrac12)-\At_k(t;\tfrac12)=-\frac{\Xi_k(\tfrac12)\,t^{k}}{2k}
		\qquad(k\ \text{odd}).
	\]
	What distinguishes them is parity in $t$.  At $p=\tfrac12$ the exact symmetry
	\eqref{eq:sym-exact} is simply $t\mapsto-t$, and by \eqref{eq:Bhat-parity} the coefficients
	$\Ah_k(\cdot\,;\tfrac12)$ are \emph{even} functions of $t$, hence manifestly invariant; the
	$\At_k$ are not, and the missing asymmetry is carried by the size-bias prefactor
	$x/(x-\rho)$.  (For instance $\At_1=-2t^{2}+2t-\tfrac14$ while $\Ah_1=-2t^{2}-\tfrac14$.)
\end{itemize}
For a general $p$ neither collapse is available, and the oscillating Bernoulli polynomial
$B_{k+1}(h_N)$ remains part of the coefficient structure.  In this precise sense the present theorem is the common
generalisation of \cite{be-binom} and \cite{paper22}.

\section{Averaging the oscillation}\label{sec:cesaro}

The coefficients do not converge.  It is natural to ask for their mean value, and the answer is
supplied by the same Bernoulli machinery --- indeed by the multiplication theorem
\begin{equation}\label{eq:mult-bernoulli}
	B_j(bt)=b^{\,j-1}\sum_{i=0}^{b-1}B_j\Bigl(t+\frac ib\Bigr),
\end{equation}
which in \cite[\S6]{paper22} served to \emph{remove} the Bernoulli polynomials in the
integer-parameter theory, and which here tells us exactly how much they contribute on average.

A word on which coefficients we average.  The pure and symmetric normalisations of
\S\ref{sec:three} carry exact, non-polynomial prefactors ($Np/m$, respectively
$(2\pi m(N-m)/N)^{-1/2}$), which themselves depend on $h$; averaging their multiplicative
coefficients is possible but the statement is cluttered.  We therefore average the coefficients of
the \emph{naive} normalisation \eqref{eq:old}, whose prefactor $(2\pi Npq)^{-1/2}$ is free of $h$,
so that the entire $h$-dependence sits in the coefficients.  This is a matter of convenience only:
the argument below is a statement about polynomials in $h$ expanded in the Bernoulli basis, and it
applies verbatim to the multiplicative coefficients of \eqref{eq:pure} or \eqref{eq:sym} once the
respective prefactor has also been expanded.

Normalise, then, as in \eqref{eq:old} and write
\begin{equation}\label{eq:cn}
	\sqrt{2\pi Npq}\;\Pr\{X=\nu+r\}\sim\sum_{n\ge0}\frac{c_n(h+r;p)}{N^{n}},
	\qquad c_0=1,
\end{equation}
the $c_n$ being the exponential polynomials in $A_1,\dots,A_n$ of \eqref{eq:Ak}.  Each $c_n$ is a
polynomial in $h$ of degree $2n$, hence has a unique expansion in the Bernoulli basis,
\begin{equation}\label{eq:cn-basis}
	c_n(h+r;p)=\sum_{j=0}^{2n}\gamma_{n,j}(r;p)\,B_j(h).
\end{equation}
We call $\gamma_{n,0}(r;p)$ the \emph{smooth part} and the rest the \emph{lattice part}.

\begin{proposition}[Ces\`aro means]\label{prop:cesaro}
	Let $0<p<1$ and $h_N=\lceil Np\rceil-Np$.
	\begin{itemize}
	\item[(i)] If $p$ is irrational, then $(h_N)_{N\ge1}$ is equidistributed in $[0,1)$ and, for
		every $n$ and every fixed $r\in\Z$,
		\[
			\lim_{M\to\infty}\frac1M\sum_{N=1}^{M}c_n(h_N+r;p)=\gamma_{n,0}(r;p)
			=\int_0^1 c_n(h+r;p)\,dh .
		\]
	\item[(ii)] If $p=a/b$ with $\gcd(a,b)=1$, then $h_N$ is periodic in $N$ with period $b$ and
		takes each of the values $0,\tfrac1b,\dots,\tfrac{b-1}b$ exactly once per period.  The
		mean of $c_n$ over a period is obtained from \eqref{eq:cn-basis} by the substitution
		\[
			B_j(h)\;\longmapsto\;b^{-j}B_j\qquad(j\ge0),
		\]
		where $B_j=B_j(0)$, so that in particular $B_1=-\tfrac12$.
	\end{itemize}
\end{proposition}

\begin{proof}
	(i) By Weyl's theorem, the sequence $(\{Np\})_{N\ge1}$ is equidistributed in $[0,1)$ for
	irrational $p$.  The map $u\mapsto1-u$ preserves Lebesgue measure, and changing the single
	endpoint value that occurs when $u=0$ is irrelevant for an irrational rotation; hence
	$h_N=\lceil Np\rceil-Np$ is equidistributed as well.  Since $c_n(h+r;p)$ is a polynomial in
	$h$, it is continuous on $[0,1]$, so equidistribution gives
	\[
		\lim_{M\to\infty}\frac1M\sum_{N=1}^{M}c_n(h_N+r;p)
		=\int_0^1c_n(h+r;p)\,dh.
	\]
	Using the Bernoulli-basis expansion \eqref{eq:cn-basis} and
	$\int_0^1B_j(h)\,dh=0$ for every $j\ge1$, the integral is exactly the constant coefficient
	$\gamma_{n,0}(r;p)$.

	(ii) With $p=a/b$ and $\gcd(a,b)=1$, the residue $Na\bmod b$ runs over all residues modulo
	$b$ as $N$ runs over a period.  Since
	\[
		h_N=\lceil Na/b\rceil-Na/b=\frac{(-Na)\bmod b}{b},
	\]
	each of the values $0,1/b,\dots,(b-1)/b$ occurs exactly once per period.  Therefore the period
	mean of a Bernoulli-basis term is
	\[
		\frac1b\sum_{i=0}^{b-1}B_j(i/b).
	\]
	Applying \eqref{eq:mult-bernoulli} with $t=0$ gives
	$\sum_{i=0}^{b-1}B_j(i/b)=b^{1-j}B_j(0)$, and hence the mean is
	$b^{-j}B_j(0)=b^{-j}B_j$.  Linearity then gives the stated substitution in
	\eqref{eq:cn-basis}.
\end{proof}

\begin{example}\label{ex:c1}
	For $n=1$ we have $c_1=A_1(h+r;p)$, and its Bernoulli-basis expansion is
	\[
		\gamma_{1,0}=\frac1{12}-\frac{r^{2}}{2pq}-\frac{r+\tfrac12}{p},
		\qquad
		\gamma_{1,1}=-\frac{r+q}{pq},
		\qquad
		\gamma_{1,2}=-\frac{1}{2pq} .
	\]
	Hence for irrational $p$ the first correction oscillates about
	\[
		\frac{1}{N}\left[\frac1{12}-\frac{r^{2}}{2pq}-\frac{2r+1}{2p}\right],
	\]
	while for $p=a/b$ its mean over a period is
	\[
		\gamma_{1,0}-\frac{\gamma_{1,1}}{2b}+\frac{\gamma_{1,2}}{6b^{2}}
		=\frac1{12}-\frac{r^{2}}{2pq}-\frac{2r+1}{2p}+\frac{r+q}{2b\,pq}-\frac{1}{12\,b^{2}pq}.
	\]
	Note $\gamma_{1,1}\ne0$ for every integer $r$: the lattice part is genuinely present, in
	accordance with Remark~\ref{rem:B1}.
\end{example}

Proposition~\ref{prop:cesaro} closes the circle with Section~\ref{sec:checks}.  In the
integer-parameter theory of \cite{be-binom} the multiplication theorem makes the Bernoulli
polynomials disappear from the final formula.  In the present, genuinely off-lattice setting they
cannot disappear --- but the same theorem says exactly how much they contribute on average,
namely $b^{-j}B_j$ per period, and nothing at all in the irrational limit $b\to\infty$.  The
smooth part $\gamma_{n,0}$ is what a naive, non-lattice calculation would produce; the lattice
part is the arithmetic correction.

\begin{remark}[Two caveats, and a comparison]\label{rem:cesaro-caveat}
	The averaging above is legitimate because the expansion \eqref{eq:cn} is uniform in $h$: the
	remainder after $M$ terms is $O(N^{-M-1})$ with a constant independent of $h_N\in[0,1)$
	(this is the uniformity of \cite[Theorem 3.1]{paper22}), so averaging a fixed truncation over
	$N$ commutes with the error.  Two things are \emph{not} claimed.  The rate in part~(i) is not
	uniform in $p$: Weyl equidistribution gives no rate without a Diophantine hypothesis on $p$, so
	the convergence $\frac1M\sum_{N\le M}c_n(h_N+r;p)\to\gamma_{n,0}$ can be arbitrarily slow.  And
	the averaged quantity is a Ces\`aro mean over $N$, not an expansion valid for a single large
	$N$.

	This should be compared with the coverage calculations of Brown, Cai and
	DasGupta~\cite{bcd}, who meet the same oscillation in the binomial confidence-interval problem.
	Their variable $g(p,z)=\lceil Np+z\sqrt{Npq}\,\rceil-Np-z\sqrt{Npq}$ is our $h_N$ at $z=0$, and
	their oscillating coefficients are, in our notation, $-\bigl[B_1(g)+\dots\bigr]$ and
	$-\tfrac12\bigl[B_2(g)+\dots\bigr]$ --- Bernoulli polynomials of the same fractional part,
	though not named as such.  They too remove the oscillation by averaging, but they average over
	\emph{$p$}, against a smooth prior, and obtain only that the oscillatory contribution is of
	lower order (a ``very weak'' expansion in the sense of Woodroofe).  Here the averaging is over
	\emph{$N$}, and it yields not merely negligibility but the exact mean: zero for irrational $p$,
	and the definite constant $b^{-j}B_j$ per period for $p=a/b$.  That constant is the point.
\end{remark}

\section{Numerical verification}\label{sec:numerics}

All computations were carried out at $50$--$60$ significant digits.

\medskip\noindent
\textbf{(a) The rate.}  We first check the order of the remainder with $t$ held \emph{fixed},
using the gamma-defined mass
$\Gamma(N+1)\big/\bigl(\Gamma(Np+t+1)\Gamma(Nq-t+1)\bigr)\cdot p^{Np+t}q^{Nq-t}$, which requires
no integrality and so allows $t$ to be frozen while $N$ grows.  The entries are the ratios of
consecutive relative errors of \eqref{eq:pure} truncated after $M$ terms, as $N$ runs through
$200,800,3200,12800$; the prediction is $4^{M+1}$.

\begin{center}\small
\begin{tabular}{llcccc}
\hline
$p$ & $t$ & $M=1$ & $M=2$ & $M=3$ & $M=4$\\
\hline
$0.3$ & $0.4$  & $15.9,\,16.0,\,16.0$ & $64.2,\,64.1,\,64.0$ & $252.9,\,255.2,\,255.8$ & $1031,\,1026,\,1024$\\
$0.3$ & $3.4$  & $15.6,\,15.9,\,16.0$ & $62.8,\,63.7,\,63.9$ & $250.0,\,254.5,\,255.6$ & $1000,\,1018,\,1023$\\
$1/\sqrt7$ & $-2.6$ & $16.5,\,16.1,\,16.0$ & $64.8,\,64.2,\,64.1$ & $262.4,\,257.6,\,256.4$ & $1043,\,1029,\,1025$\\
\hline
\multicolumn{2}{l}{\emph{predicted}} & $16$ & $64$ & $256$ & $1024$\\
\hline
\end{tabular}
\end{center}

\noindent
The symmetric form \eqref{eq:sym} passes the same test with the same rates.  In addition,
$b(m;N,p)$ computed from \eqref{eq:sym} agrees with $b(N-m;N,q)$ computed from \eqref{eq:sym} to
$61$ digits, confirming that \eqref{eq:Ahsym} is an exact symmetry of the coefficients and not
merely an asymptotic one.

\medskip\noindent
\textbf{(b) At the true lattice points.}  Relative error of \eqref{eq:main} for
$\Pr\{X=\lceil Np\rceil+r\}$, with $p=1/\sqrt7$, so that $h_N$ genuinely oscillates:

\begin{center}\small
\begin{tabular}{rrcccc}
\hline
$r$ & $N$ & $M=1$ & $M=2$ & $M=3$ & $M=4$\\
\hline
$0$ & $200$  & $4.1\cdot10^{-7}$  & $6.3\cdot10^{-9}$  & $1.7\cdot10^{-11}$ & $2.9\cdot10^{-13}$\\
$0$ & $3200$ & $2.5\cdot10^{-10}$ & $1.8\cdot10^{-12}$ & $4.1\cdot10^{-17}$ & $3.2\cdot10^{-19}$\\
$1$ & $3200$ & $5.7\cdot10^{-8}$  & $3.3\cdot10^{-11}$ & $7.1\cdot10^{-15}$ & $2.7\cdot10^{-18}$\\
$3$ & $3200$ & $1.9\cdot10^{-6}$  & $4.5\cdot10^{-9}$  & $4.6\cdot10^{-12}$ & $9.1\cdot10^{-15}$\\
\hline
\end{tabular}
\end{center}

\medskip\noindent
\textbf{(c) The window.}  For $p=0.3$, $R=2$, $N=1600$ the exact value of
$\Pr\{|X-\nu|\le2\}$ is $0.1084764352$.  The leading term of Corollary~\ref{cor:window} alone
overshoots by a factor $1.00317$; including the $N^{-1}$ term gives the ratio $0.99999$.

\medskip\noindent
\textbf{(d) Structural identities.}  The size-bias identity \eqref{eq:sizebias} holds to
$10^{-51}$.  The splittings \eqref{eq:split} are identities in
$\mathbb Q\bigl[h,p^{-1},q^{-1}\bigr]$ for $1\le k\le5$ and $-3\le r\le4$, in both the $S_k$ and
the $T_k$ version.  The asymmetry \eqref{eq:asym} and the symmetry \eqref{eq:Ahsym} are
identities for $1\le k\le6$, and the parity \eqref{eq:Bhat-parity} for $n\le7$.  The mode rule of
Proposition~\ref{prop:mode} has no exception as a statement about the upper mode in $2388$
exact rational tests
($p\in\{\tfrac{17}{100},\tfrac3{10},\tfrac12,\tfrac57,\tfrac{43}{50},\tfrac13\}$,
$2\le N\le399$).  The Ces\`aro means of Proposition~\ref{prop:cesaro} were confirmed for
$n=1,2$: exactly, over a period, for $p\in\{\tfrac13,\tfrac25,\tfrac37,\tfrac58\}$ and
$r\in\{-2,-1,0,1,2\}$; and numerically, over $M=2\cdot10^{5}$ values of $N$, for
$p\in\{1/\sqrt7,\ \pi-3,\ \sqrt2-1\}$.  The coefficients $P_n$ of \S\ref{sec:checks} agree with
\cite{be-binom} exactly for $n\le4$.

\section{Concluding remarks}

	The elementary tail of \cite[Theorem 3.1]{paper22} is $-\log(m/(Np))$; the De~Moivre prefactor
	is the size-bias factor; and the cancellation observed there is precisely the identity
	\eqref{eq:sizebias}.  The right normalisation of the local mass is
	$\bigl(m/(Np)\bigr)b(m;N,p)$, equivalently $b(m-1;N-1,p)$, and that object is pure Appell.

	The three natural normalisations separate the roles of purity and symmetry.  The size-biased
	normalisation is governed by the Bernoulli Appell sequence and is pure but not
	$p\leftrightarrow q$ symmetric; the entropy normalisation is governed by the even sequence
	$\Bh_n$ generated by $\tfrac z2\coth\tfrac z2$ and restores the classical Stirling symmetry.
	The naive normalisation is governed by neither, its coefficients being a Bernoulli part plus an
	elementary tail.  The obstruction between the two Appell forms is $\Xi_k(p)\,t^{k}/(2k)$.

	The integer shift $r$ contributes only a power sum, and the Appell sequence decides which one:
	the left-endpoint rule in the size-biased form, and the trapezoidal rule in the symmetric form.
	This is the Euler--Maclaurin correspondence behind the $r$-dependence.

	Two extensions remain natural.  The first is the regime $r\to\infty$ with $N$, especially the
	crossover at $r\asymp\sqrt N$ into the Edgeworth regime, where the relevant object should be the
	uniform (error-function) normalisation of the incomplete beta function.  The second is the same
	analysis for other lattice distributions with an Appell structure --- Poisson, negative
	binomial, hypergeometric --- where a size-bias identity again exists and one may ask which
	Appell sequence replaces $B_n$ and $\Bh_n$.

\backmatter

\section*{Declarations}

\textbf{Funding.}\enspace The author did not receive support from any organization
for the submitted work.


\end{document}